\documentclass[final,1p,number,11pt]{elsarticle}

\usepackage{amsmath,amssymb,amsthm}
\DeclareMathOperator{\RE}{Re} 
\numberwithin{equation}{section}
\newtheorem{theorem}{Theorem}[section]
\newtheorem{lemma}{Lemma}[section]
\newtheorem{corollary}{Corollary}[section]
\theoremstyle{remark}
\newtheorem{remark}{Remark}[section]
\newtheorem{example}{Example}[section]
\makeatletter
\newcommand{\subjclassname@later}{\textup{2010} Mathematics Subject
Classification} \makeatother

\journal{Computers \& Mathematics with Applications}

\begin{document}

\begin{frontmatter}

\title{Inclusion Criteria for Subclasses of Functions and Gronwall's Inequality \tnoteref{t1}}

\tnotetext[t1]{The work presented here was  supported in part by a grant  from Universiti Sains Malaysia. This work was completed during the
last two authors' visit to Universiti Sains Malaysia.}

\author[rma]{Rosihan M. Ali\corref{cor1}}
\address[rma]{School of Mathematical Sciences,
Universiti Sains Malaysia, 11800 USM, Penang, Malaysia}
\ead{rosihan@cs.usm.my}

\author[rma]{Mahnaz M. Nargesi}
\ead{mmn.md08@student.usm.my}

\author[rma,vr]{V. Ravichandran\fnref{fn3}}
\address[vr]{Department of Mathematics, University of Delhi,
Delhi 110 007, India} \ead{vravi@maths.du.ac.in}

%

\author[swami]{A. Swaminathan}
\address[swami]{Department of Mathematics,
Indian Institute of Technology Roorkee, Roorkee 247 667, India}
\ead{swamifma@iitr.ernet.in}

\cortext[cor1]{Corresponding author}

\begin{abstract}
A normalized analytic function $f$ is  shown to be univalent in the open unit disk $\mathbb{D}$ if its second coefficient  is sufficiently small
and relates to its Schwarzian derivative through a certain inequality. New criteria for
 analytic functions  to be in certain subclasses of
 functions are established in terms of the Schwarzian derivatives
and the second coefficients. These include obtaining a sufficient condition for functions to be strongly $\alpha$-Bazilevi\v c of order $\beta$.

\end{abstract}

\begin{keyword}
Univalent functions, Bazilevi\v c functions, Gronwall's inequality, Schwarzian derivative, second coefficient.

\MSC[2010]   30C45
\end{keyword}

\end{frontmatter}

\section{Introduction} Let $\mathcal{A}$ be the set of all
normalized analytic functions $f$ of the form $f(z) = z + \sum_{k=2}^{\infty}a_k z^k $ defined in the open unit disk $\mathbb{D}:=\{z\in
\mathbb{C}: |z|<1\}$ and denote by $\mathcal{S}$ the subclass of $\mathcal{A}$ consisting of  univalent functions. A function $f\in \mathcal{A}$
is starlike if it maps $\mathbb{D}$ onto a starlike domain with  respect to the origin, and $f$ is convex if $f(\mathbb{D})$ is a convex domain.
Analytically, these are respectively equivalent to the conditions $\RE(zf'(z)/f(z))>0$ and $1+\RE(zf''(z)/f'(z))>0$ in $\mathbb{D}$. Denote by
$\mathcal{ST}$ and $\mathcal{CV}$ the classes of starlike and convex functions respectively. More generally, for $0\leq \alpha <1$, a function
$f\in \mathcal{A}$ is starlike of order $\alpha$ if $\RE(zf'(z)/f(z))>\alpha$, and  is convex of order $\alpha$ if
$1+\RE(zf''(z)/f'(z))>\alpha$. We denote these classes by $\mathcal{ST}(\alpha)$  and $\mathcal{CV}(\alpha)$ respectively. For $0 <\alpha \leq
1$,  let $\mathcal{SST}(\alpha)$ be the subclass of $\mathcal{A}$  consisting of functions $f$  satisfying the inequality
\begin{align*}
 \left|\arg
\dfrac{zf'(z)}{f(z)}\right| \leq \frac{\alpha\pi}{ 2}.
\end{align*}Functions in $\mathcal{SST}(\alpha)$ are called strongly starlike functions of order $\alpha$.\\

 The  Schwarzian derivative $S(f,z)$  of a locally univalent analytic function $f$ is
defined by
\[
S(f,z):=  \left(\frac{f''(z)}{f'(z)}\right)' - \frac{1}{2}
\left(\frac{f''(z)}{f'(z)}\right)^2.
\]
The Schwarzian derivative is invariant under M\"obius
transformations. Also, the Schwarzian derivative of an analytic
function $f$ is identically zero  if and only if it is a M\"obius
transformation.

Nehari   showed that the univalence of an analytic function in ${\mathbb{D}}$ can be guaranteed if its Schwarzian derivative is dominated by a
suitable positive function \cite[Theorem I, p.\ 700]{N2}. In  \cite{N1}, by considering two particular positive functions, a bound on the
Schwarzian derivative was obtained that would ensure univalence of an analytic function in $\mathcal{A}$. In fact, the following theorem was
proved.

\begin{theorem}\cite[Theorem II, p.\ 549]{N1}
If $f\in {\mathcal{A}}$  satisfies
\[
|S(f,z)|\leq \dfrac{\pi^2}{2} \quad (z\in {\mathbb{D}}),\]then
$f\in \mathcal{S}$.  The result is sharp for the function $f$ given
by $f(z)=(\exp(i\pi z)-1)/i\pi$.
\end{theorem}

The problems of finding  similar bounds on the Schwarzian derivatives that would imply univalence, starlikeness or convexity of functions were
investigated by a number of authors including Gabriel \cite{gaber}, Friedland and Nehari \cite{F}, and Ozaki and Nunokawa \cite{ Oza}.
Corresponding results  related to  meromorphic functions were dealt with in \cite{gaber, Haimo, N1, P}. For instance, Kim and Sugawa
\cite{sugawa} found sufficient conditions in terms of the Schwarzian derivative for locally univalent meromorphic functions in the unit disk to
possess specific geometric properties such as starlikeness and convexity. The method of proof in \cite{sugawa} was based on  comparison theorems
in the theory of ordinary differential equations with real coefficients.

 Chiang \cite{chiang} investigated  strong-starlikeness of order
$\alpha$ and convexity of functions $f$ by requiring the Schwarzian
derivative $S(f ,z)$ and the second coefficient $a_2$ of $f$ to
satisfy certain inequalities. The following results were proved:

\begin{theorem} \cite[Theorem 1, pp.\ 108-109]{chiang} \label{ch:th1}
Let $f\in \mathcal{A}$, $0<\alpha\leq 1$ and
$|a_2|=\eta<\sin(\alpha\pi/2)$. Suppose
\begin{equation}\label{ch:eq1}
\sup_{z\in\mathbb{D}}| S(f,z)| =2\delta(\eta),
\end{equation}
where $\delta(\eta)$ satisfies the inequality
\[\sin^{-1}\left(\frac{1}{2}\delta e^{\delta/2}\right) +
\sin^{-1}\left(\eta+\frac{1}{2}(1+\eta)\delta e^{\delta/2}\right)
\leq \frac{\alpha\pi}{2}.  \] Then $f\in \mathcal{SST}(\alpha)$.
Further, $|\arg(f(z)/z)| \leq \alpha\pi/2$.
\end{theorem}

\begin{theorem}\cite[Theorem 2, p.\ 109]{chiang}\label{thm:suff-cv-chiang}
Let $f\in \mathcal{A}$, and $|a_2|=\eta<1/3$. Suppose \eqref{ch:eq1}
holds where $\delta(\eta)$ satisfies the inequality
\[
6\eta+5(1+\eta)\delta e^{\delta/2} < 2.
\]
Then
\[ f\in\mathcal{CV}\left( \frac{2-6\eta-5(1+\eta)\delta e^{\delta/2} }{
2-2\eta-(1+\eta)\delta e^{\delta/2} } \right).
\]
In particular, if $a_2=0$ and $2\delta\leq 0.6712$, then
$f\in\mathcal{CV}$.
\end{theorem}

Chiang's proofs in \cite{chiang}  rely  on Gronwall's inequality (see Lemma \ref{gronineq} below). In this paper, Gronwall's inequality is used
to obtain sufficient conditions for analytic functions to be univalent. Also, certain inequalities related to the Schwarzian derivative and the
second coefficient will be formulated that would ensure analytic functions to possess certain specific geometric properties. The sufficient
conditions of convexity obtained in \cite{chiang}   will be seen to be a special case of our result, and similar conditions for starlikeness
will also be obtained.

\section{Consequences of Gronwall's  Inequality }    Gronwall's
inequality and certain relationships between the Schwarzian derivative of $f$ and the solution of the linear second-order differential equation
$y''+A(z)y=0$ with $A(z):=S(f;z)/2$ will be revisited in this section.  We first state \textit{Gronwall's inequality}, which is needed in our
investigation.

\begin{lemma}\cite[p.\ 19]{gron}\label{gronineq}
Suppose  $A$ and $g$ are non-negative continuous real functions for
$t\geq 0$. Let $k> 0$ be a constant. Then the inequality
\[ g(t) \leq k+\int_0^t g(s)A(s)ds  \]
implies
\[ g(t) \leq k \exp\left( \int_0^t A(s)ds \right)  \quad ( t>0). \]
\end{lemma}

For the linear second-order differential equation $y''+A(z)y=0$
where $A(z):=\frac{1}{2} S(f;z)$ is an analytic function, suppose
that $u$ and $v$ are two linearly independent solutions  with
initial conditions $u(0)=v'(0)=0$ and $u'(0)=v(0)=1$. Such solutions
always exist and thus the function   $f$ can be represented by
\begin{equation} \label{f}
f(z)= \frac{u(z)}{c u(z)+v(z)}, \quad (c:=-a_2).
\end{equation}
 It is evident   that
\begin{equation} \label{fd}
 f'(z)= \frac{1}{(c u(z)+v(z))^2}.
\end{equation}
 Estimates on bounds for various expressions related to $u$ and $v$
 were found in \cite{chiang}. Indeed, using the integral
 representation of the fundamental solutions
\begin{equation}\label{*}
  \begin{array}{ll}
  u(z) & = z + \int_0^z (\eta-z) A(\eta)u(\eta)d\eta ,\\[10pt]
   v(z) & = 1 + \int_0^z (\eta-z) A(\eta)v(\eta)d\eta,
  \end{array}
\end{equation}
 and applying Gronwall's inequality, Chiang obtained the following inequalities   \cite{chiang}
  which we list for easy reference:
\begin{align}\label{eqn-mod-u}
 &|u(z)|  < e^{\delta/2}, \\
\label{eqn-mod-u-z}
 &\left|\frac{u(z)}{z}-1\right|  <\frac{1}{2}\delta e^{\delta/2}, \\
\label{eqn-mod-cu+v}
&|cu(z)+v(z)|    < (1+\eta)e^{\delta/2}, \\
\label{eqn-mod-cu+v-1}
 &|cu(z)+v(z)-1|   < \eta + \frac{1}{2}(1+\eta)\delta e^{\delta/2}.
 \end{align}
 For instance,  by
 taking the path of integration  $\eta(t)=te^{i\theta}$,
 $t\in[0,r]$, $z=re^{i\theta}$, Gronwall's inequality
 shows that, whenever $|A(z)|<\delta$ and $0<r<1$,
\begin{align*}
|u(z)| & \leq 1+\int_0^r(r-t)|A(te^{i\theta})|\,|u(te^{i\theta})|dt \\
& \leq \exp(\int_o^r (r-t)|A(te^{i\theta})|dt) \leq \exp(\delta/2).
\end{align*}  This proves   inequality \eqref{eqn-mod-u}.
Note  that there was  a typographical error in \cite[ Inequality (8), p.\ 112]{chiang}, and that  inequality  \eqref{eqn-mod-u-z} is the right
form.

\section{Inclusion Criteria for   Subclasses of Analytic  Functions}
The first result leads to  sufficient conditions for univalence.

\begin{theorem}{\label{Thm-f'}}Let  $0<\alpha\leq1$,  $0\leq\beta<1$,
 $f\in \mathcal{A}$ and $|a_2|=\eta $,  where
$\alpha$,  $\beta$ and $\eta$ satisfy
\begin{equation}\label{Q14a} \sin^{-1}\big(\beta(1+\eta)^2\big)
+2 \sin^{-1}\eta<\frac{\alpha \pi}{2}.\end{equation}
 Suppose \eqref{ch:eq1} holds
where $\delta (\eta )$ satisfies the inequality
\begin{equation}\label{Q14}
\sin^{-1}\big(\beta(1+\eta)^2 e^{\delta}\big)+2
\sin^{-1}\left(\eta+\frac{1}{2}(1+\eta)\delta
e^{\delta/2}\right)\leq  \frac{\alpha\pi}{2}
\end{equation}
Then  $|\arg(f'(z)-\beta) | \leq \alpha \pi/2$.
\end{theorem}

\begin{proof}Using a limiting argument as $\delta\rightarrow 0$, the condition \eqref{Q14a} shows that there is a real
number $\delta(\eta)\geq0$  satisfying  inequality \eqref{Q14}. The
representation of $f'$ in terms of the linearly independent
solutions of the differential equation $y''+A(z)y=0$ with $ A(z):=
S(f;z)/2$ as given by
             equation  \eqref{fd} yields
\begin{align}
f'(z)-\beta =\frac{1-\beta(c\ u(z)+v(z))^2}{(c\
u(z)+v(z))^2}\label{th1e0}.
\end{align}
In view of the fact that for $w\in \mathbb{C}$, \[|w-1|\leq r\Leftrightarrow|\arg w|\leq \sin^{-1} r, \]
  inequality $(\ref{eqn-mod-cu+v})$ implies
\begin{align}\label{th1e2}
|\arg[1-\beta(c\ u(z)+v(z))^2]|\leq \sin^{-1}\big(\beta(1+\eta)^2
e^{\delta}\big).
\end{align} Similarly,  inequality \eqref{eqn-mod-cu+v-1} shows
\begin{equation} \label{th1e3}
|\arg[c\ u(z)+v(z)]| \leq  \sin^{-1}\left(\eta+\frac{1}{2}(1+\eta)\delta e^{\delta/2}\right).
\end{equation}
Hence,  it follows from \eqref{th1e0}, \eqref{th1e2} and \eqref{th1e3}  that
\begin{align*}
|\arg(f'(z)-\beta)|&\leq|\arg[1-\beta(c\ u(z)+v(z))^2]|+2|\arg[c\ u(z)+v(z)]|\\
&\leq\sin^{-1}(\beta(1+\eta)^2 e^{\delta})+2
\sin^{-1}\left(\eta+\frac{1}{2}(1+\eta)\delta e^{\delta/2}\right)\\
&\leq  \frac{\alpha\pi}{2},
\end{align*}
where the last inequality follows from \eqref{Q14}. This completes  the
proof.
\end{proof}

By taking $\beta=0$ in Theorem \ref{Thm-f'}, the following univalence criterion is obtained.

\begin{corollary}\label{Thm-suff-univ}
Let $f\in \mathcal{A}$, and  $|a_2|=\eta < \sin (\alpha\pi /4)$,
$0<\alpha\leq1$. Suppose $(\ref{ch:eq1})$ holds where $\delta (\eta
)$ satisfies the inequality
\begin{equation*}
 \eta+\frac{1}{2}(1+\eta)\delta e^{\delta/2} \leq \sin \left(\frac{\alpha\pi}{4}\right).
\end{equation*}
Then $|\arg f'(z)|\leq \alpha\pi /2$, and in particular
$f\in\mathcal{S}$.
\end{corollary}

\begin{example}\label{example-g} Consider the univalent function $g$ given by
\[g(z)=\frac{z}{1+cz},\ \quad |c|\leq 1, \quad z\in {\mathbb{D}}.\]
Since the Schwarzian derivative of an analytic function is  zero if
and only if it is a M\"obius transformation, it is evident that $S(g,z)=0$. Therefore
the condition \eqref{ch:eq1} is satisfied with $\delta=0$.
 It is enough to take $\eta=|c|$ and to assume that
$\eta $, $\alpha $ and $\beta$ satisfy the inequality \eqref{Q14a}.
Now
\begin{align*}|\arg (g'(z)-\beta)|&=\left|\arg \frac{1}{(1+cz)^2}-\beta\right|
\leq|\arg(1-\beta (1+cz)^2)|+2|\arg (1+cz)|\\
&\leq\sin^{-1}(\beta(1+|c|)^2)+2 \sin^{-1}|c|.\end{align*} In view of the latter inequality, it is necessary to assume  inequality \eqref{Q14a}
for $g$ to satisfy $|\arg(g'(z)-\beta) | \leq \alpha \pi/2$.
\end{example}
Let  $0\leq\rho< 1$, $0\leq\lambda<1$, and $\alpha$ be a positive integer. A function $f\in\mathcal{A}$ is called an $\alpha$-Bazilevi\v c
function of order $\rho$ and type $\lambda$, written $f\in \mathcal{B}(\alpha,\rho,\lambda)$, if \[ \RE \left(
\frac{zf'(z)}{f(z)^{1-\alpha}g(z)^\alpha}\right)>\rho\quad (z\in\mathbb{D})\] for some function $g\in \mathcal{ST}(\lambda)$. The following
subclass of $\alpha$-Bazilevi\v c functions  is of interest.  A function $f\in \mathcal{A}$ is called strongly $\alpha$-Bazilevi\v c of order
$\beta$ if
\begin{align*}\label{eqn-Bazil}
\left|\arg\left(\left(\frac{z}{f(z)}\right)^{1-\alpha}f'(z)\right)\right|
<\frac{\beta\pi}{2},\quad   (\alpha>0;\  0<\beta\leq 1),\end{align*}
(see Gao \cite{Gao}). For the class of strongly $\alpha$-Bazilevi\v
c functions of order $\beta$,  the following sufficient condition is
obtained.

\begin{theorem}\label{Thm: suff-bazil} Let $\alpha>0$,  $0<\beta\leq
1$, $f\in \mathcal{A}$ and $|a_2|=\eta $,  where $\eta$, $\alpha$
and $\beta$ satisfy \[
\eta<\sin\left(\frac{\beta\pi}{2(1+\alpha)}\right).\]  Suppose
\eqref{ch:eq1} holds where $\delta (\eta )$ satisfies the inequality
\begin{equation}\label{6}
|1-\alpha|\sin^{-1}\left(\frac{1}{2}\delta e^{\delta/2}\right)
+(1+\alpha )\sin^{-1}\left(\eta+\frac{1}{2}(1+\eta)\delta e^{\delta/2}\right)
\leq\frac{\beta\pi}{2}.
\end{equation}
Then   $f$ is strongly $\alpha$-Bazilevi\v c of order $\beta$.
\end{theorem}

\begin{proof}The condition $\eta<\sin(\beta\pi/2(1+\alpha))$ ensures that there is a real
number $\delta(\eta)$ satisfying \eqref{6}.  Using \eqref{f} and \eqref{fd}  lead to
\begin{align*}
\left|\arg\left(\left(\frac{z}{f(z)}\right)^{1-\alpha}f'(z)\right)\right|
& =\left| \arg\left( \left(\frac{u(z)}{z}\right)^{\alpha-1}  \left(cu(z)+v(z)\right)^{-(\alpha+1)}  \right)\right|\\
&\leq |1-\alpha|\left|\arg\left(\frac{u(z)}{z}\right)\right|
+|\alpha+1|\,\left|\arg(cu(z)+v(z))\right|.
\end{align*}
It now follows from  \eqref{eqn-mod-u-z},  \eqref{th1e3} and
\eqref{6}  that
\begin{align*}
\left|\arg\left(\left(\frac{z}{f(z)}\right)^{1-\alpha}f'(z)\right)\right|& \leq |1-\alpha|\sin^{-1}\left(\frac{1}{2}\delta e^{\delta/2}\right)
+(1+\alpha )\sin^{-1}\left(\eta+\frac{1}{2}(1+\eta)\delta e^{\delta/2}\right)\\
&\leq\frac{\beta\pi}{2}. \qedhere
\end{align*}
\end{proof}

For $\alpha\geq0$, consider the class $R(\alpha)$ defined by
\[\mathcal{R}(\alpha)=\{f\in \mathcal{A}: \RE  \left(f'(z)+\alpha zf''(z)\right)>0, \ \alpha\geq0\}.\]
For this class,   the following  sufficient condition is obtained.

\begin{theorem}\label{Thm-libera(f)}Let $\alpha \geq 0$, $f\in
\mathcal{A}$ and $|a_2|=\eta $,  where $\eta$ and $\alpha$ satisfy
\begin{equation}\label{4}
2\sin^{-1}\eta+ \sin^{-1}\left(\frac{2\eta\alpha}{1-\eta}\right)<\frac{\pi}{2}.\end{equation}
 Suppose \eqref{ch:eq1} holds
where $\delta (\eta )$ satisfies the inequality
\begin{equation}\label{5}
2 \sin^{-1}\left(\eta+\frac{1}{2}(1+\eta)\delta
e^{\delta/2}\right)+\sin^{-1}\left(\frac{4\alpha\big(\eta+(1+\eta)\delta
e^{\delta/2}\big)}{2-2\eta-(1+\eta)\delta
e^{\delta/2}}\right)\leq\frac{\pi}{2}.
\end{equation}
Then  $ f\in\mathcal{R}(\alpha) $.
\end{theorem}

\begin{proof} Again it is easily seen from a limiting argument that the condition \eqref{4} guarantees the existence of
  a  real number $\delta(\eta)\geq0$   satisfying the
inequality \eqref{5}. It is sufficient to show  that
\begin{align*}
\left|\arg\left(f'(z)\left(1+\alpha\frac{zf''(z)}{f'(z)}\right)\right)\right|<\frac{\pi}{2}.
\end{align*}
 The equation  \eqref{fd} yields
\begin{align}\label{ksp1a}
\frac{zf''(z)}{f'(z)}&=-2z \frac{c u'(z)+v'(z)}{cu(z)+v(z)}.
\end{align}
A simple calculation from  \eqref{*} shows that
\[ cu'(z)+v'(z)= c-\int_0^zA(\eta)[cu(\eta)+v(\eta)]d\eta, \]
and  an application of  \eqref{eqn-mod-cu+v} leads to
\begin{align}\label{ksp3}
|cu'(z)+v'(z)|\leq \eta+(1+\eta)\delta e^{\delta/2}.
\end{align}
Use of   \eqref{eqn-mod-cu+v-1} results in
\begin{align}\label{ksp2}
|cu(z)+v(z)| \geq 1-|cu(z)+v(z)-1| \geq 1- \eta-\frac{1}{2}(1+\eta)\delta e^{\delta/2}.
\end{align}
The lower bound in \eqref{ksp2} is non-negative from the assumption made in \eqref{5}.  From \eqref{ksp1a}, \eqref{ksp3} and \eqref{ksp2} , it
is evident that
\begin{align}\notag
\left|\left(1+\alpha\frac{zf''(z)}{f'(z)}\right)-1\right|& =\left|2z\alpha\frac{c u'(z)+v'(z)}{c
u(z)+v(z)}\right| \\
&   \leq \frac{2\alpha\big(\eta+(1+\eta)\delta
e^{\delta/2}\big)}{1-\eta-\frac{1}{2}(1+\eta)\delta
e^{\delta/2}}.\notag
\end{align}
Hence,
\begin{equation}\label{Q2}
\left|\arg \left(1+\alpha \frac{zf''(z)}{f'(z)}\right)\right|
\leq\sin^{-1}\left(\frac{4\alpha\big(\eta+(1+\eta)\delta
e^{\delta/2}\big)}{2-2\eta-(1+\eta)\delta e^{\delta/2}}\right).
\end{equation}
From \eqref{th1e3} it follows that
\begin{align}\label{ksp1}
|\arg f'(z)|=2|\arg(cu(z)+v(z))|\leq 2
\sin^{-1}\left(\eta+\frac{1}{2}(1+\eta)\delta e^{\delta/2}\right).
\end{align} Using  \eqref{fd} and \eqref{th1e3},
the inequality \eqref{ksp1} together with \eqref{Q2}  and \eqref{5}
imply that
\begin{align*}
\left|\arg \left( f'(z)\left(1+\alpha\frac{zf''(z)}{f'(z)}\right)\right) \right|
&\leq |\arg f'(z)|+\left|\arg\left(1+\alpha \frac{zf''(z)}{f'(z)}\right)\right|\\
&\leq 2 \sin^{-1}\left(\eta+\frac{1}{2}(1+\eta)\delta e^{\delta/2}\right)
    +\sin^{-1}\left(\frac{4\alpha\big(\eta+(1+\eta)\delta e^{\delta/2}\big)}{2-2\eta-(1+\eta)\delta e^{\delta/2}}\right)\\
    &\leq\frac{\pi}{2}.\qedhere
\end{align*}
\end{proof}
\begin{theorem}\label{thm-suff-nonlinear-combination}
Let $f\in \mathcal{A}$, $|a_2|=\eta \leq1/3$,  and $\beta$, $\alpha$ be real numbers satisfying
\begin{equation}\label{9}|\alpha|\sin^{-1}\eta+
|\beta|\sin^{-1}\left(\frac{2\eta}{1-\eta}\right)<\frac{\pi}{2}.\end{equation}
 Suppose \eqref{ch:eq1} holds
where $\delta (\eta )$ satisfies the inequality
\begin{equation}\label{8}\begin{split}
 |\alpha|\sin^{-1}\left(\frac{1}{2}\delta e^{\delta/2}\right)
+|\alpha|\sin^{-1}\left(\eta+\frac{1}{2}(1+\eta)\delta e^{\delta/2}\right) \\
{}+|\beta|\sin^{-1}\left(\frac{4\big(\eta+(1+\eta)\delta
e^{\delta/2}\big)}{2-2\eta-(1+\eta)\delta
e^{\delta/2}}\right)\leq\frac{\pi}{2}.
\end{split}
\end{equation}
Then
\begin{align}\label{eqn-nonlinear-st-cv}
{\rm \RE\,}
\left(\left(\dfrac{zf'(z)}{f(z)}\right)^{\alpha}\left(1+\dfrac{zf''(z)}{f'(z)}\right)^{\beta}\right)>0.
\end{align}
\end{theorem}

\begin{proof} The inequality \eqref{9} assures the existence of
  $\delta$ satisfying \eqref{8}.  From \eqref{f} and
\eqref{fd} it follows that
\begin{align}\label{eq-zf'/f}
\dfrac{z f'(z)}{f(z)}=\dfrac{z}{u(z)}\dfrac{1}{cu(z)+v(z)}, \quad z\in {\mathbb{D}}.
\end{align}
By \eqref{eqn-mod-u-z} and  \eqref{th1e3},
\begin{equation}\label{3}\left|\arg\left(\frac{zf'(z)}{f(z)}\right)\right|\leq\sin^{-1}\left(\frac{1}{2}\delta
e^{\delta/2}\right)+\sin^{-1}\left(\eta+\frac{1}{2}(1+\eta)\delta
e^{\delta/2}\right).\end{equation} Using \eqref{Q2} with $\alpha=1$,
\eqref{3} and \eqref{8}  lead to
\begin{align*}
\left|\arg\left(\left(\frac{zf'(z)}{f(z)}\right)^{\alpha}
\left(1+\frac{zf''(z)}{f'(z)}\right)^{\beta}\right)\right|
&\leq|\alpha|\left|\arg\left(\frac{zf'(z)}{f(z)}\right)\right|
+|\beta|\left|\arg\left(1+\frac{zf''(z)}{f'(z)}\right)\right|\\
&\leq|\alpha|\sin^{-1}\left(\frac{1}{2}\delta
e^{\delta/2}\right)+|\alpha|\sin^{-1}\left(\eta+\frac{1}{2}(1+\eta)\delta
e^{\delta/2}\right)\\
&  \quad{}+|\beta|\sin^{-1}\frac{4\big(\eta+(1+\eta)\delta
e^{\delta/2}\big)}{2-2\eta-(1+\eta)\delta e^{\delta/2}}\\
&\leq\frac{\pi}{2}.\end{align*} This shows that
\eqref{eqn-nonlinear-st-cv} holds.
\end{proof}

\begin{remark}
Theorem \ref{thm-suff-nonlinear-combination} yields
the following interesting special cases.
\begin{enumerate}[(i)]
\item If $\alpha=0$,  $\beta=1$, a sufficient condition for
convexity is obtained. This case reduces to a result in
\cite[Theorem 2, p.\ 109]{chiang}.

\item For $\alpha=1$,  $\beta=0$, a sufficient condition for
starlikeness  is obtained.

\item For  $\alpha=-1$ and $\beta=1$, then the class of functions
satisfying \eqref{eqn-nonlinear-st-cv} reduces to the class of
functions
\[ \mathcal{G}:= \left\{f\in \mathcal{A}\Bigg| \ {\rm \RE\,}
\left(\dfrac{1+\frac{zf''(z)}{f'(z)}}{\frac{zf'(z)}{f(z)}}\right)>0\right\}.\]
This class $\mathcal{G}$  was  considered   by Silverman \cite{Sil}
and Tuneski \cite{Tuneski}.
\end{enumerate}

\end{remark}

\begin{theorem}
Let $\beta\geq0$, $f\in \mathcal{A}$ and $|a_2|=\eta $, where $\eta$
satisfies
\begin{align}\label{lasteqn} \sin^{-1}\left(\eta\right)
+\sin^{-1}\left(\frac{2\beta\eta}{1-\eta }\right)<\frac{\pi}{2}.
\end{align}
Suppose \eqref{ch:eq1} holds
where $\delta (\eta )$ satisfies the inequality
\begin{align*}
 \sin^{-1}\left(\frac{1}{2}\delta e^{\delta/2}\right)
+\sin^{-1}\left(\eta+\frac{1}{2}(1+\eta)\delta e^{\delta/2}\right)
+\sin^{-1}\left(\frac{4\beta\big(\eta+(1+\eta)\delta
e^{\delta/2}\big)}{2-2\eta-(1+\eta)\delta
e^{\delta/2}}\right)\leq\frac{\pi}{2}.
\end{align*}
Then
\begin{align}\label{kspcon}
 \RE\left( \dfrac{zf'(z)}{f(z)}+  \beta \dfrac{z^2f''(z)}{f(z)}\right)>0.
\end{align}
\end{theorem}
The proof is similar to the proof of Theorem
$\ref{thm-suff-nonlinear-combination}$, and is therefore omitted.
The inequality \eqref{lasteqn} is equivalent to the condition \[
\eta\left( 1+\sqrt{(1-\eta)^2-4\beta^2\eta^2}+2\beta\sqrt{1-\eta^2}
\right) <1.
\]
For $\beta=1$, the above equation simplifies to
\begin{align*}
\eta^8-4 \eta^7+12 \eta^6-12 \eta^5+6 \eta^4+20 \eta^3-4 \eta^2-4
\eta+1 = 0;
\end{align*} the value of the root $\eta$ is  approximately  0.321336.
Functions    satisfying   inequality \eqref{kspcon}
were investigated  by Ramesha \emph{et  al.}\ \cite{KSP}.\\

Consider the class $P(\gamma)$, $0\leq\gamma\leq 1$, given by
\[
P(\gamma):= \left\{f\in {\mathcal{A}}:
\left|\arg\left((1-\gamma)\dfrac{f(z)}{z}+\gamma f'(z)\right)\right|<\dfrac{\pi}{2}, \quad z \in {\mathbb{D}}\right\}.
\]
The same approach applying Gronwall's inequality leads   to the
following result about the class $P(\gamma)$.

\begin{theorem}\label{thm-p-gamma} Let $0\leq \gamma <1$, $f\in {\mathcal{A}}$ and $|a_2|=\eta$, where $\eta$ and $\gamma$ satisfy
\begin{equation}\label{7}
\sin^{-1}\left(\dfrac{\gamma}{1-\gamma}\dfrac{1}{\eta-1}\right)+\sin^{-1} \eta<\dfrac{\pi}{2}.
\end{equation}
Suppose $(\ref{ch:eq1})$ holds
where $\delta(\eta)$ satisfies the inequality
\begin{equation}\begin{split}\label{eq-p-gamma}
&\sin^{-1}\left(\dfrac{1}{2}\delta e^{\delta/2}\right)
+\sin^{-1}\left(\eta+\dfrac{1}{2}(1+\eta)\delta e^{\delta/2}\right)\\
&+ \sin^{-1}\left(\dfrac{2\gamma}{1-\gamma}\dfrac{1}{2-2\eta-(1+\eta)\delta
 e^{\delta/2}}\dfrac{1}{1-2e^{\delta/2}}\right)\leq\dfrac{\pi}{2}.
\end{split}
\end{equation}
Then $f\in P(\gamma)$.
\end{theorem}

\begin{proof} Condition \eqref{7} assures the existence of  a real
number $\delta(\eta)\geq0$ satisfying the inequality \eqref{eq-p-gamma}. A simple calculation from \eqref{*}  and Lemma \ref{gronineq} shows
that
\begin{align*}
|u(z)-1|& \leq |z-1|+\left|\int_0^z (\zeta-z)A(\zeta)u(\zeta)d\zeta \right| \\
& \leq 2 \displaystyle e^{\delta/2}.
\end{align*}
The above inequality gives
\begin{align}\label{eq-z/u}
\left|\dfrac{z}{u(z)}\right| \leq \dfrac{1}{|u(z)|}\leq \dfrac{1}{1-|u(z)-1|}\leq \dfrac{1}{1-2e^{\delta/2}}.
\end{align}
Therefore, for some $0<\beta\leq\gamma/(1-\gamma)$, \eqref{eq-zf'/f}, \eqref{eq-z/u} and \eqref{ksp2} lead to
\begin{align*}
\left|1+\beta \dfrac{zf'(z)}{f(z)}-1\right|
& = \beta \left|\dfrac{z}{u(z)}\right|\, \dfrac{1}{\left|cu(z)+v(z)\right|}\\
& \leq \dfrac{\beta}{1-2e^{\delta/2}}\dfrac{1}{1-\eta-\frac{1}{2}(1+\eta)\delta e^{\delta/2}}\\
&= \dfrac{2\beta}{1-2e^{\delta/2}}\dfrac{1}{2-2\eta-(1+\eta)\delta e^{\delta/2}}.
\end{align*}
Hence
\begin{equation}\label{1}
\left| \arg\left(1+\beta \dfrac{zf'(z)}{f(z)}\right)\right|
\leq\sin^{-1}\left(\dfrac{2\beta}{1-2e^{\delta/2}}\dfrac{1}{2-2\eta-(1+\eta)\delta
e^{\delta/2}}\right).
\end{equation}
Also,   \eqref{eqn-mod-u-z} and \eqref{th1e3} yield
\begin{align}\label{2}
\left|\arg \dfrac{f(z)}{z}\right|&=\left|\arg \dfrac{u(z)}{z(cu(z)+v(z))}\right|\notag \\
& \leq
\left|\arg \dfrac{u(z)}{z}\right|+ \left|\arg (cu(z)+v(z))\right|\notag\\
&\leq \sin^{-1}\left(\dfrac{1}{2}\delta e^{\delta/2}\right)+
\sin^{-1} \left(\eta+\dfrac{1}{2}(1+\eta)\delta e^{\delta/2}\right).
\end{align}
Replacing $\beta$ by $\gamma/(1-\gamma)$ in  inequality \eqref{1},  and using \eqref{2} and \eqref{eq-p-gamma} yield
\begin{align*}
\left|\arg\left((1-\gamma)\dfrac{f(z)}{z}+\gamma f'(z)\right)\right|
&\leq \left|\arg\dfrac{f(z)}{z}\right|
 + \left|\arg\left(1+\dfrac{\gamma}{1-\gamma} \dfrac{zf'(z)}{f(z)}\right)\right| \\
& \leq \sin^{-1}\left(\dfrac{1}{2}\delta e^{\delta/2}\right)
+\sin^{-1}\left(\eta+\dfrac{1}{2}(1+\eta)\delta e^{\delta/2}\right)\\
&\quad{}
+\sin^{-1}\left(\dfrac{2\gamma}{1-\gamma}\dfrac{1}{1-2e^{\delta/2}}
\dfrac{1}{2-2\eta-(1+\eta)\delta e^{\delta/2}}\right)\\
&\leq\dfrac{\pi}{2},
\end{align*}
and hence $f\in P(\gamma)$.
\end{proof}

 \end{document}